\documentclass[10pt,a4paper,english]{article}
\usepackage{showlabels}
\usepackage[colorlinks=true]{hyper ref}
\usepackage[T1]{fontenc}
\usepackage[latin1]{inputenc}
\usepackage{extarrows}
\usepackage{latexsym,amsmath,amsthm,bbm}  
\usepackage{pgfplots}
\usepackage{pgfplotstable}
\pgfplotsset{compat=1.8}
\usepackage{mleftright}
\usepackage{booktabs}
\usepackage{mathtools}
\usepackage{graphicx}

\usepackage{amssymb}
\usepackage{caption}
\usepackage{subcaption}
\usepackage{tikz}
\usetikzlibrary{shapes,shapes.geometric, intersections}

\makeatletter
\renewcommand{\vec}[1]{\boldsymbol{#1}}
\providecommand{\boldsymbol}[1]{\mbox{\boldmath $#1$}}
\makeatother

\newtheorem{theorem}{Theorem}[section]  
\newtheorem{lemma}[theorem]{Lemma}

\newtheorem{proposition}[theorem]{Proposition}
  
\newtheorem{remark}[theorem]{Remark}

\newcommand{\opL}{\mathcal{L}}
\newcommand{\R}{\mathbb{R}}
\newcommand{\C}{\mathbb{C}}
\newcommand{\N}{\mathbb{N}}

\newcommand{\Sph}{\mathbb{S}}
\newcommand{\m}{m_0}
\newcommand{\M}{\mathbf{M}}
\newcommand{\K}{\mathbf{K}}

\newcommand{\Nn}{\mathcal{H}_{\varPhi}}
\newcommand{\Samp}{\varSigma}
\newcommand{\Id}{\operatorname{Id}}

\newcommand{\dist}{\operatorname{dist}}

\newcommand{\spn}{\operatorname{span}}

\newcommand{\PhiB}{\mathbf{\Phi}}
\newcommand{\PsiB}{\mathbf{\Psi}}
\newcommand{\dif}{\mathrm{d}}

\newcommand{\Manoa}{M\=anoa}
\newcommand{\Hawaii}{Hawai\kern.05em`\kern.05em\relax i }

\title{Highly Localized RBF Lagrange Functions  for Finite Difference Methods on Spheres}
\author{ W.~Erb 
\thanks{Dipartimento  di Matematica ``Tullio Levi-Civita'', Universit{\`a} degli Studi di Padova, 
Via Trieste 63, 35121 Padova, Italy (erb@math.unipd.it).},
 T.~Hangelbroek
\thanks{Department of Mathematics, University of \Hawaii   -- \Manoa,
 Honolulu, HI 96822, USA (hangelbr@math.hawaii.edu).   
Research supported by  by grants DMS-1716927  and DMS-2010051 from the National
    Science Foundation.},
    F. J.~Narcowich\thanks{ Department of Mathematics, Texas A\&M
    University, College Station, TX 77843, USA (fnarc@math.tamu.edu). Research
    supported by grant DMS-1813091 from the National
    Science Foundation.}, 
    C.~Rieger\thanks{Philipps-Universit\"at Marburg, 
    Department of Mathematics and Computer Science,
Hans-Meerwein-Stra\ss{}e 6, 35032 Marburg, Germany (riegerc@mathematik.uni-marburg.de).},
J. D.~Ward\thanks{ Department of Mathematics, Texas A\&M University,
    College Station, TX 77843, USA (jward@math.tamu.edu). Research supported by
    grant DMS-1813091  from the National Science
    Foundation.}} 
\date{\today}                                           

\begin{document}

\maketitle
\begin{abstract}
The aim of this paper is to show how 
rapidly decaying RBF Lagrange functions on the 
spheres can be used to create effective, stable finite difference methods based on radial basis functions (RBF-FD). 
For certain classes of PDEs this approach leads to precise convergence estimates for stencils which grow moderately with 
increasing discretization fineness. 
\end{abstract}
\section{Introduction}
The RBF-FD method is a modification of the classical finite difference method suitable for working with unstructured point sets.
Instead of using mesh-based finite difference operators which have some fixed degree of polynomial exactness, 
one enforces exactness on finite dimensional spaces generated by kernels of radial basis functions. 
In its simplest
form, this means finding a matrix, the {\em RBF-FD matrix},
which represents the differential operator on a finite dimensional
kernel space.
 
To illustrate this setup, we consider a time-independent partial differential equation
on a manifold without boundary (such as the sphere $\Sph^2$)
$$\opL u = f.$$
The PDE is replaced by  a linear system
 $$\M \tilde{u} = y$$
 which can then be solved for a discrete solution $\tilde{u}$.
The RBF-FD matrix 
$\M\in \R^{M\times N}$ encodes $\opL$ on a finite dimensional space $V_{X_N} = \mathrm{span}_{x_j\in X} k(\cdot,x_j)$
generated by a kernel $k$. 
The vector $y=f|_{Y_M}$ is obtained from $f$ by sampling at a discrete set of  points $Y_M$.
A suitable approach will have $M=\#Y_M\ge N=\#X_N$, although one may consider square systems with $Y_M=X_N$, which is what we
do throughout this paper.

 We note that this completely avoids numerical quadrature,
 and, by working coordinate free, is algorithmically straightforward.
 So as a  method, this  is 
 a
 computationally  efficient method of solving differential equations on 
manifolds.
\paragraph{Stability, consistency and convergence}
As with classical finite difference methods, 
convergence in this setting is ensured by 
 consistency of the method and stability of the matrix $\M$.
In many cases, the use of positive definite kernels gives an automatic and satisfying consistency theory.
The basic theoretical challenge is that the RBF-FD matrices are, in general,  not even known to be  invertible, much
less stably invertible.

This poses a substantial challenge for time dependent problems as well, 
where $\M$ may be  used to generate a system of ODEs 
which can then be solved 
with classical integrators. We will not focus on those issues here, although  the problems we address
in this paper (namely stability and invertibility of the underlying systems) are also well-known in this context.
 There  exist a number of creative modifications  to attack instability, including 
oversampling and regularization,
careful grid and parameter selection,
hyperviscosity methods   \cite{chu,tominec_hyperbolic,tominec,shankar}.

A major goal of this paper is to give conditions on the operator $\opL$ which guarantee stable invertibility of $\M$, 
and provide satisfying convergence rates. 

 \paragraph{Sparse systems}
 Another challenge involves constructing sparse RBF-FD matrices by having exactness
only on a (very small, spatially dependent) subspace of the full kernel space; i.e., by making use of
a ``stencil'' which is very local in space. 

\medskip

{\em Full stencils} This involves constructing a global matrix $\M$ which represents $\opL$ 
on the $N$-dimensional space $V_{X_N}$.
For each $y_j\in Y_M$, the corresponding  row of $\M$  is full. 
This method has theoretically advantages, but is less common in practice,
due to the high computational cost of constructing $\M$ and working with it.  It has the advantage of being highly consistent, 
with consistency bounds resulting from the well-developed theory of kernel interpolation. 
(See section \ref{SS:Consistency} for a precise discussion of this.)

\smallskip

{\em Small stencils} This involves construction of a very sparse matrix $\M^{\circ}$,
 whose rows are determined by representing $\opL$ at a point, using only the (low dimensional)
space generated by the immediate neighbors of the point.
More precisely, exactness at $y_j$ is determined on
$\mathrm{span}_{x_k\in \mathcal{N}( y_j)} k(\cdot,x_k)$, where $\mathcal{N}(y_j)$ is a set of nearby neighbors of $y_j$.
This has the advantage of driving down the computational cost for constructing the RBF-FD matrix -
roughly $\mathcal{O}(N n^3)$ operations if for each $j$,  $\# \mathcal{N}( y_j)=n$.
This may make the resulting linear systems easier to solve, but at the  cost of 
the theoretical consistency bounds.
A generally open question is how sparse to make such a system, while providing suitably
good convergence rates. 

\smallskip


\paragraph{Focus of this paper}

This paper addresses both of the above challenges in the context 
of  the sphere $\mathbb S^2$. 
In this case, the kernels employed are called SBFs (spherical basis functions), and 
we naturally refer to the SBF-FD method (in place of the RBF-FD method).
Although we do not treat the issue of small stencils directly, we consider
a closely related challenge, which is the second focus of this paper: the challenge of easing 
computational complexity of the problem by providing a sparse alternative to the full-stencil RBF-FD matrix. 
The issue of sparsity is treated by employing  rapidly decaying RBF Lagrange functions,
which  have been constructed and thoroughly analyzed in 
 \cite{FHNWW,HNRW1,HNRW2}.

 \paragraph{Main contributions}
 We provide the first (to our knowledge) positive RBF-FD  results in the classical stability/consistency framework (there exist other
approaches which treat convergence directly -- see \cite{DAVYDOV2019}).
 A (measured) stability result for Helmholtz operators
 is given in section \ref{S:stability}. 
 
 Additionally, we  consider  stencils obtained 
from local Lagrange functions. 
This leads to rapidly constructed FD matrices $\M^{\sharp}$ which, although not as sparse as $\M^{\circ}$, 
 enjoy rapid off-diagonal decay.
Moreover, these perform as well as the full FD matrices $\M$  created  directly from RBFs themselves. 
When applied to Helmholtz operators, we demonstrate that
both full and local stencil schemes using 
 kernels with localized Lagrange functions
provide $L_2$ convergence.

 \smallskip

\smallskip
\paragraph{Overview}
The paper is organized as follows. 
Section 2 provides necessary background for analysis on the sphere and positive
definite ``spherical basis functions'' (SBFs).
In section 3, we describe the
the domain independent 
 SBF-FD setup.
Here we discuss the construction of general
 SBF-FD matrices,
their potential singularity via a basic factorization, and then the consistency of this method.
Section 4 treats the challenge of stability. 
We provide stability
 result for a specific class of differential operators (Helmholtz operators),
 which is stronger than anything we have seen in the literature, and
 substantial enough to use in our convergence results.
  Section 5 introduces the  local Lagrange basis and the associated local SBF-FD method.
  Here we show that the local SBF-FD matrices $\M^{\sharp}$ are, up to a constant, as stable as the full-stencil analogues $\M$.
Section  6 contain the main  convergence result of 
the paper.
\section{Background on kernel methods}\label{Background}

\subsection{The Sphere}
We denote $\Sph^2 = \{x\in \R^2\mid |x|= 1\}$. The distance function on the sphere is  
$\dist(x_1,x_2) := \arccos (x_1\cdot x_2)$ for $x_1,x_2\in \Sph^2$. 
The basic neighborhood is $B(x,r):=\{y\in\Sph^2: \dist(x,y)<r\}$; 
it has volume given by the formula $\mu(B(x,r)) = 2\pi(1-\cos r)$.

The sphere has the usual spherical coordinate 
parametrization by $\theta$ and $\varphi$,
with $x = \cos \theta \sin \varphi$, $y= \sin\theta \sin \varphi$ and $z=\cos \varphi$.
Lebesgue measure is $\dif \mu= \sin(\varphi)\dif\theta \dif \varphi$.

The Laplace-Beltrami operator is
$\Delta = \frac{1}{\sin^2(\varphi)}\frac{\partial^2}{\partial \theta^2} 
+\frac{1}{\sin(\varphi)}\frac{\partial}{\partial  \varphi}\sin(\varphi) \frac{\partial}{\partial \varphi}$.
It is self-adjoint, negative semi-definite.

For each $\ell\in \N$,
$\nu_\ell:=-\ell(1+\ell)$ is an eigenvalue of  the Laplace-Beltrami operator.
The corresponding eigenspace 
has an orthonormal
basis of $2\ell+1$ eigenfunctions,
$\{Y_{\ell}^\mu\}_{\mu=-\ell}^{\ell}$.
called {\em spherical harmonics} 
of degree $\ell$ defined by
\begin{equation}\label{spherical_harmonic}
Y_{\ell}^\mu(x):=
\begin{cases}
a_{\ell}^{\mu}P_{\ell}^{\mu}(\cos\varphi)\cos(\mu \theta)&\mu\ge 0\\
a_{\ell}^{|\mu|}P_{\ell}^{|\mu|}(\cos \varphi)\sin(\mu \theta)&\mu<0.
\end{cases}
\end{equation}
Here $P_{\ell}^{|\mu|}$ is the $\mu$th associated Legendre polynomial of degree $\ell$, and $a_{\ell}^{|\mu|}$ is a normalization factor.
The space of spherical harmonics of degree
$\ell\le M$ is denoted $\Pi_M := \mathrm{span}\{Y_{\ell}^{\mu}\mid |\mu|\le \ell, \ell\le M\}$ 
and has dimension $(M+1)^2$. 
Using the definition (\ref{spherical_harmonic}), the collection $(Y_{\ell}^{\mu})_{\ell\ge 0,|\mu|\le \ell}$ forms an orthonormal basis for $L_2(\Sph^2)$. 

\paragraph{Sobolev spaces} 
The Sobolev space $H^m(\Sph^2)$ is defined as $\{f\in L_2(\Sph^2)\mid \|f\|_m<\infty\}$ where
the norm $\|f\|_m$ is induced from the inner product $(f,g)\mapsto \langle f,g\rangle_m$, which is defined  for $f= \sum a_{\ell,\mu} Y_{\ell}^{\mu}$ and $g= \sum b_{\ell,\mu} Y_{\ell}^{\mu}$ as
$$\langle f,g\rangle_m 
:= 
\sum_{\ell=0}^\infty \sum_{|\mu|\le \ell} {a_{\ell,\mu} b_{\ell,\mu}}{(1+|\nu_{\ell}|)^m}.$$ 
The space $H^m(\Sph^2)$ is a Hilbert space; for $m>1$, we have
the continuous embedding $H^m(\Sph^2)\subset C(\Sph^2)$.

\paragraph{Point sets} 
For $\Omega\subset \Sph^2$ and finite subset $X\subset \Omega$, 
we define the {\em fill distance} of $X$ in $\Omega$
as $$h:= \max_{x\in \Omega} \dist(x,X) = \max_{x\in \Omega}\min_{x_j\in X} \dist(x,x_j).$$
The separation radius is 
$$q:= \frac12\min_{x_j\in X}\min_{x_k\neq x_j}\dist(x_j,x_k).$$
Throughout the paper, we consider quasi-uniformly distributed point sets $X$, i.e., point sets for which the mesh ratio $\rho:= h/q$ is bounded (although $h$ and $q$ may be very small). Specifically,
we assume there is a constant $\rho_0$ so that for any $X\subset \Sph^2$, $\rho\le \rho_0$.

\subsection{Positive definite kernels}
Let $\varPhi:\Sph^2 \times \Sph^2 \to \R$ be a continuous positive definite kernel, meaning that for every
finite set 
$X_N=\{x_1,\dots,x_N\}\subset \Sph^2$
(with $\#X_N = N$) the 
collocation matrix 
$$
\PhiB_{X_N} 
:=\bigl(\varPhi(x_j,x_k)\bigr)_{j,k=1\dots N} =
\begin{pmatrix}  \varPhi(x_{1},x_{1}) & \dots & \varPhi(x_{1},x_{N}) \\
\vdots & \ddots & \vdots  \\
 \varPhi(x_{N},x_{1}) & \dots &  \varPhi(x_{N},x_{N})
\end{pmatrix} 
$$
is strictly positive definite (see \cite{NW} for background). 

There is an associated  reproducing kernel Hilbert  space, called the {\em native space},
$\mathcal{H}_{\varPhi}(\Sph^2) \subset C(\Sph^2)$ with inner product 
$(f,g)\mapsto \langle f,g\rangle_{\varPhi}$
for which $\varPhi$ is the reproducing kernel: for all $f\in  \mathcal{H}_{\varPhi}(\Sph^2)$, 
 $f(x) = \langle f,\varPhi(x,\cdot)\rangle_{\varPhi}$.
It follows that for each $N$-set $X_N$, the  linear space
\begin{align}\label{trialspace}
S({X_{N}}) 
:=\spn\left\{\varPhi(\cdot,x_{j}) \ : \ x_{j} \in X_{N} \right\} \subset \mathcal{H}_{\varPhi}(\Sph^2)
\subset C(\Sph^2),
\end{align} 
is $N$-dimensional.

The kernels we will consider in this article are spherical basis functions (SBFs), which have the form
$\varPhi(x,y) = \phi(x\cdot y)$ for some continuous univariate function $\phi:[-1,1]\to \R$; 
these can emerge, for instance, as the restriction to $\Sph^2$ 
of a radial basis function (RBF -- a translation, rotation invariant kernel) on $\R^3$. 
They have a Hilbert-Schmidt expansion 
\begin{equation}\label{HS}
\varPhi(x,y) =\sum_{\ell=0}^{\infty} \sum_{m=-\ell}^{\ell} \widehat{\phi}(\ell) Y_{\ell}^m(x) Y_{\ell}^m (y)
\end{equation}
 with positive coefficients $\widehat{\phi}(\ell)$.  
 For $f = \sum a_{\ell,\mu} Y_{\ell}^{\mu}$ and $g= \sum b_{\ell,\mu} Y_{\ell}^{\mu}$, 
 the inner product on $\mathcal{H}_{\varPhi}(\Sph^2)$
 is 
 $$\langle f,g\rangle_{\mathcal{H}_{\varPhi}} = 
 \sum_{\ell=0}^{\infty} \sum_{|\mu|\le \ell} \frac{a_{\ell,\mu} \overline{b_{\ell,\mu}}}{\widehat{\phi}(\ell)} .$$
 See Appendix \ref{SS:shape} for a number or concrete examples of positive definite SBFs.

We identify now two important maps associated with the point set $X_N$.
The first is simply the sampling operator (restriction to $X_N$) which we denote 
$\varSigma_{X_N}$, so $\varSigma_{X_N}f := f\left|_{X_N}\right.$; 
it is continuous as a map from $C(\Sph^2)$ 
(and hence $\mathcal{H}_{\Phi}$) to $\C^{N}$.
The interpolation operator 
$
{I}_{X_N} 
: C(\Sph^2)\to S({X_{N}})$
 maps $f$ to its unique $X_N$ interpolant in $S(X_N)$.
Namely, ${I}_{X_N}f  = \sum_{j=1}^N a_{j} \varPhi(\cdot, x_j) $ 
where $\vec{a}=(a_j)_{j=1\dots N}$
is the unique solution to $\PhiB_{X_N}\vec{a}  = \Sigma_N f$.

When restricted to $\mathcal{H}_{\varPhi}\subset C(\Sph^2)$, 
the interpolation operator ${I}_{X_N} $
is the orthogonal
projector onto $S({X_{N}})$ (with respect to the $\mathcal{H}_{\varPhi}(\Sph^2)$ inner product).
This is the basis for a number of theoretical results, like error estimates and stability bounds.
In particular, if $\mathcal{H}_{\varPhi} = H^m(\Sph^2)$ is a Sobolev space, 
then for all $f\in H^m(\Sph^2)$, we have
$$\|f- I_{X_N} f\|_{L_p(\Sph^2)} \le C h^{m-(\frac12-\frac1p)_+} \|f\|_{H^{m}(\Sph^2)}.$$
Stronger results are possible: namely,
\cite[Theorem 5.5]{NSWW} 
shows that if $1<\beta\le m$ and $\tau\le \beta$ then
there is a constant $C$ so that 
for all $f\in H^m(\Sph^2)$
$$\|f- I_{X_N} f\|_{H^\tau(\Sph^2)} \le C h^{\beta-\tau} \|f\|_{H^{\beta}(\Sph^2)}.$$

\subsection{Conditionally positive definite kernels}
A modest change to the SBF theory presented in Section \ref{Background} is possible,
by  relaxing the requirement of strict positive definiteness. 
 A kernel $\varPhi:\Sph^2\times\Sph^2\to \R$ is {\em conditionally positive definite} if 
 for any reasonable point set $X_N$, the matrix $\PhiB_{X_N} = (\varPhi(x_j,x_k))_{j,k}$ 
  is positive definite  on the complement of a certain space of (small) fixed dimension. 
 This space of fixed codimension, and thereby the conditional positive definiteness, is  precisely 
 described using low order spherical harmonics as follows: 
 for  nonzero vectors $\vec{a}\in \C^{N}$ satisfying the ``moment conditions''
\begin{equation}\label{side_condition}
 (\forall \ell \le \tilde{m}),\ (\forall |\mu|\le \ell)\  \sum_{j=1}^N a_j Y_{\ell}^{\mu}(x_j)=0
 \end{equation}
the quadratic form induced by $\PhiB_{X_N}$
is strictly positive: i.e., $\vec{a}^T \PhiB_{X_N} \vec{a} > 0$  for all $\vec{a}\neq 0$ which satisfy (\ref{side_condition}).
In this case, we say that $\varPhi$ is conditionally positive definite of order $\tilde{m}$.

On a  practical level, in order to solve interpolation problems with $\varPhi$
we need only to augment 
the collocation matrix $\PhiB_{X_N}$ by 
a Vandermonde style matrix
 $\mathbf{P}= \bigl(P_k (x_j)\bigr)_{j\le N, k\le N_{\tilde{m}}}
\in \R^{N\times N_{\tilde{m}}}$, where $\{P_k\mid 1\le k \le N_{\tilde{m}}\}$ is a basis for $\Pi_{\tilde{m}}$
(we may take $\mathbf{P} = \bigl(Y_{\ell}^{\mu}(x_j)\bigr)$, although any other choice of basis will suffice).
Here
$N_{\tilde{m}} := (\tilde{m}+1)^2=\dim \Pi_{\tilde{m}}$ is the
 dimension of spherical harmonics of degree $\tilde{m}$ or less.
The augmented interpolation matrix  
$$\tilde{\PhiB}_{X_N}
:=
\begin{pmatrix} \PhiB_{X_N} &\mathbf{P}\\ \mathbf{P}^T &\mathbf{0}\end{pmatrix}
$$
is  a nonsingular, square matrix of width 
$ N+N_{\tilde{m}}$.
The solution of the system
 $\tilde{\PhiB}_{X_N} \begin{pmatrix}A\\B\end{pmatrix}= \begin{pmatrix}y\\0\end{pmatrix}$
  provides coefficients for the unique element of the $N$-dimensional space
 $$S(X_N,\tilde{m}) 
 := \left\{\sum_{j=1}^N A_j \varPhi(\cdot, x_j)
 \, \middle|\,
 (\forall P\in \Pi_{\tilde{m}}),\ 
 \sum_{j=1}^N A_j P(x_j)=0\right\}
 +
 \Pi_{\tilde{m}}. 
 $$
\begin{remark}
 In many cases, the order is fixed with the kernel, and one writes $S(X_N)$ in lieu of $S(X_N,\tilde{m})$.
 \end{remark}

\begin{remark}\label{SPD_is_CPD}
Any (strictly) positive definite kernel  is conditionally positive definite (of any order).
Similarly, a conditionally positive definite kernel of order $\m$ is is conditionally positive definite of order $m'\ge \m$ as well.
To ease the exposition, we adopt the convention that a conditionally positive definite SBF of order $\m=-1$ is strictly positive definite.
\end{remark}
 There is a native space theory for the conditional positive definite setup: the space $\mathcal{H}_{\varPhi}$
 is a reproducing kernel semi-Hilbert space, with semi-inner product $\langle f,g\rangle_{\mathcal{H}_{\varPhi}} = \sum_{\ell>\tilde{m}}\sum_{|\mu|\le \ell} a_{\ell,\mu} \overline{b_{\ell,\mu}} \bigl(\widehat{\phi}(\ell)\bigr)^{-1}$ for $f = \sum a_{\ell,\mu} Y_{\ell}^{\mu}$ and $g= \sum b_{\ell,\mu} Y_{\ell}^{\mu}$.
 We refer to \cite{Wend} for an introduction.

\subsubsection{Restricted thin plate splines}\label{SSS:RTPS}
The restricted thin plate splines
are a prominent class of conditionally positive SBFs. 
They have the form
$\phi(t)= (1-t)^{m-1} \log (1- t)$, leading to conditionally positive definite 
kernels $\varPhi(x,y) = \phi(x\cdot y)$, which are 
conditionally positive definite of order $\tilde{m}=m-1$.

For $\ell\ge m$, the coefficients in the Hilbert-Schmidt expansion  (\ref{HS}) of $\phi$ (see \cite[Lemma 3.4]{H-sphere})
are 
\begin{equation}
\label{TPS_Coeff}
\hat{\phi}(\ell) = C\nu_{\ell}(\nu_\ell+2)\dots \bigl(\nu_{\ell} +m(m-1)\bigr)
=
C\prod_{j=0}^{m-1} \bigl(\nu_{\ell}+j(j+1)\bigr).
\end{equation}
Here $C$ is a constant which depends on $m$, but is independent of $\ell$.
We recall that $\nu_{\ell}=-\ell(\ell+1)$ is the $\ell$th eigenvalue of $\Delta$.
It follows that, up to a constant multiple, $\Phi_m$ is the fundamental solution for 
the elliptic operator
$ \prod_{j=0}^{m-1}\bigl(\Delta -j(j+1)\bigr)$
for all functions in $\{f\in H^{2m}(\Sph^2)\mid f\perp \Pi_{m-1}\}$.


%
%
%
\section{RBF-FD matrix and its factorization}\label{Sect:factoring}
Here, we basically recall material from \cite[Section 5]{fornberg_flyer_2015}, 
which is mostly already contained in \cite{fornberg_1988}.
The main idea of RBF-FD methods is to represent the action of a differential operator 
$\opL$ 
exactly on the space $S({X_{N}})$ from \eqref{trialspace} by a linear map $\opL_{X_N}$.
I.e., to construct a linear map, such that
\begin{align}\label{exactness}
\left.\opL_{X_{N}} \right|_{S({X_{N}})} =\left.\opL \right|_{S({X_{N}})}
\end{align}
holds. 
Such a map can be defined by composition with the interpolation operator:
$
\opL_{X_{N}}  u 
:=
\opL \circ \mathcal{I}_{X_N} (u).
$
By construction, we have that \eqref{exactness} holds, i.e. $\opL_{X_{N}}$ equals exactly 
$\opL $ on $S({\Xi_{N}})$. 
The construction of this map can be visualized using the auxiliary space 
\begin{align*}
\opL \left( S({X_{N}}) \right):=
\spn\left\{\opL^{(1)} \varPhi(\cdot,x_{j}) \ : \ x_{j} \in X_{N} \right\} 
\subset 
 H^{m-\m}(\Sph^d).
\end{align*}
The RBF-FD methods consider
the linear mapping which leads to the so-called RBF-differentiation-map.
This is the matrix which represents $\opL_{X_{N}}$ on $\C^N$ in the sense that
$
\Samp_{X_{N}} \opL_{X_{N}}
=
\M_{ X_{N}} 
\Samp_{X_{N}}.
$

Consider now 
$\chi_{x_{j}}:=\chi_{x_{j};X_{N}}$, the  Lagrange function in $S({X_{N}})$
(using the usual standard bases 
$\vec{e} _{j} \in \R^{N},$ 
with 
$\vec{e}_{j}(k)=\delta_{j,k}$  
for $ 1\le j,k\le N$ ).
Each Lagrange function also can be expressed in the standard kernel basis, i.e., we have
\begin{align*}
\chi_{{x}_{j}} 
= \sum_{k = 1}^{N} A_{{j},k} \varPhi(\cdot,x_k)
\qquad
 \Rightarrow 
 \qquad 
 \opL \,
 \chi_{{x}_{j}}
 =
 \sum_{k = 1}^{N} A_{{j},k} \opL^{(1)}\varPhi(\cdot,x_k) 
\end{align*}
where  
$
A_{{j},{k}} 
$
is the $j,k$ entry of  
$(\PhiB_{X_{N}})^{-1}$.
With respect to these bases and the notation from above, 
we get that the RBF-differentiation-matrix satisfies
$\M_{X_{N}} \vec{e} _{j}  = \Samp_{X_N} \mathcal{L} \chi_{x_j}$
so the matrix has the form
\begin{align}\label{rbffd:lagrange}
\M_{X_{N}}
=
\begin{pmatrix}
\opL \chi_{x_{1}}({x}_{1}) 
& 
\dots 
& 
\opL \chi_{x_{N}}({x}_{1}) \\
\vdots & \ddots & \vdots  \\
\opL  \chi_{x_{1}}({x}_{N}) 
& 
\dots 
& 
\opL \chi_{x_{N}}({x}_{N})
\end{pmatrix} 
\in \R^{N \times N}.
\end{align}
By rewriting 
$ 
\opL \chi_{x_k}(x_j) = \sum_{\ell=1}^N  A_{{j},{\ell}} \opL^{(1)}  \varPhi(x_j, x_{\ell})
$
the matrix can  be factored as 
\begin{align*}
\M_{X_{N}} = 
\begin{pmatrix}
\opL^{(1)}  \varPhi(x_{1},{x}_{1}) & \dots & \opL^{(1)} \varPhi(x_{N},{x}_{1}) \\
\vdots & \ddots & \vdots  \\
\opL^{(1)} \varPhi(x_{1},{x}_{N}) & \dots & \opL^{(1)} \varPhi(x_{N},{x}_{N})
\end{pmatrix} 
(\PhiB_{X_N})^{-1}
=: 
\K_{X_N}\PhiB^{-1}_{X_{N}},
\end{align*}
where 
$ \K_{X_N} 
$ 
is the so-called Kansa matrix which arises in un-symmetric collocation methods. 
 \subsection{Factorization in the conditionally positive definite setting}
 In this case, we have
 \begin{align*}
 \M_{X_N}:= \bigl( \opL \chi_{j}(x_k)\bigr)_{j,k} 
 = 
 [\K_{X_N} \vert \opL\mathbf{ P}] \left[ \begin{matrix} A\\  B\end{matrix}\right],
 \end{align*}
 where the Lagrange basis has an expansion of the form
 %
 $
\chi_{x_j} 
= 
\sum_{x_k\in \Upsilon(x_j)}^N A_{j,k}\varPhi(\cdot, x_k)
 +\sum_{\ell=0}^{\tilde{m}}\sum_{\mu=-\ell}^{\ell} B_{j,\ell,\mu} Y_{\ell}^\mu
 $,
with coefficients  determined by 
 \begin{equation}\label{lagrange_coeffs}
   \begin{bmatrix} \PhiB_{X_N} &  \mathbf{P} \\ \mathbf{P}^T &\mathbf{0}_{N_{\tilde{m}}\times N_{\tilde{m}}}\end{bmatrix} 
   \left[ \begin{matrix} A\\  B\end{matrix}\right] 
   = \begin{bmatrix} \mathbf{I}_{N\times N}\\ \mathbf{0}_{N_{\tilde{m}}\times N_{\tilde{m}}}\end{bmatrix}.
   \end{equation}
As before, the FD matrix involves a Kansa-type matrix.

As we will see below, for a Helmholtz operator $\opL=\alpha-\Delta$ and $\alpha>0$,
$\K_{X_N} $ is an SBF collocation matrix. 
In this case, the auxiliary matrix $\opL \mathbf{P}$ is also a Vandermonde type matrix generated by spherical
harmonics, since $\opL$ is an isomorphism on $\Pi_{\tilde{m}}$. From a theoretical point of view, the choice of auxiliary matrix $\mathbf{P}$
is not important, but, as we will see, the stability analysis benefits from using an orthonormal basis for $\Pi_{\tilde{m}}$.

\subsection{Consistency}\label{SS:Consistency}
In order to assess the quality of the approximation, 
we first introduce the notion of \emph{consistency}, measured
by  
\begin{align}\label{consis}
\left\|\M_{X_N} (u\left|_{X_N}\right.) - (\opL u)\left|_{X_N}\right. \right\|_{\ell_p(X_N)}
=:
\rho_{u}(h)\to 0, \quad \text{as } h := \max_{X\in \Sph^2} \mathrm{dist} (x,X_N)\to 0.
\end{align}  

The consistency condition is easily verified by using estimates for kernel interpolation on  
 $d$-dimensional spheres,
which are derived from zeros estimates.
 For many kernels, 
specifically those with native spaces which are Sobolev spaces $H^{m}(\mathbb{S}^d)$, 
interpolation error can be bounded by way of Sobolev error estimates as in
  \cite[Theorem 1.1]{NWW}.
In short,  we have
$\|u - I_{X_N} u\|_{C^{\m}(\Sph^d)} \le C h^{{m}-d/2- \m} \|u\|_{H^{m}(\Sph^d)}$
when $m>\m+d/2$.
This  easily provides a satisfying consistency estimate.
We note that 
$\|\M_{X_N} (u\left|_{X_N}\right.) - (\opL u)\left|_{X_N}\right. \|_{\ell_{\infty}(X_N)}
\le \|\opL u - \opL_{X_N} u\|_{\infty}$, so
\begin{equation}\label{consistency}
\|\opL u - \opL_{X_N} u\|_{\infty} = \| \opL (u - I_{X_N} u)\|_{\infty} \le \|u - I_{X_N} u\|_{C^{\m}}
\le C h^{{m}-1 - \m} \|u\|_{H^{m}(\Sph^2)}
\  \longrightarrow \ 
\rho_u(h) =\mathcal{O}( h^{m-1-\m})
\end{equation}
for  $u\in H^{m}(\Sph^2)$.
Because the sphere lacks a boundary, we may go a step further by invoking the ``doubling trick'' 
of Schaback: for $u\in H^{2m}(\Sph^d)$, the improved estimate
$\|u - I_{X_N} u\|_{C^{\m}(\Sph^d)} \le C h^{{2m}-d/2 - \m} \|u\|_{H^{2m}(\Sph^d)}$ holds.  
This gives
an improved consistency for for  $u\in H^{2m}(\Sph^2)$:
\begin{equation}\label{double_consistency}
\|\opL u - \opL_{X_N} u\|_{\infty} 
\le \|u - I_{X_N} u\|_{C^{\m}}
\le C h^{{2m}-1 - \m} \|u\|_{H^{2m}(\Sph^2)}
\  \longrightarrow \ 
\rho_u(h) =\mathcal{O}( h^{2m-1-\m})
\end{equation}

We note that these estimates are still  somewhat pessimistic;
there is substantial evidence that the penalty $-d/2$ 
in the approximation order is 
not necessary. 
This has been shown to hold for SBF approximation (not interpolation) in (\cite{MNPW,H-sphere}),
and for $L_p$ approximation of the interpolant in \cite{HNW-p}. 
In other words, the true $C^{\m}(\Sph^2)$ rate of approximation for
SBF interpolation
is likely $\mathcal{O}(h^{2m-\m})$ for sufficiently smooth  $u$ 
(although this has not been proven).

\section{Stability}\label{S:stability}
The \emph{stability} of the approximation is measured in terms of
\begin{align}\label{stab}
\left\|\M^{-1}_{X_N}\right\|_{p\to p} 
=
\sup_{\genfrac{}{}{0pt}{}{w \in \ell_p(X_N)}{w \neq 0}}
\frac{\left\|\M^{-1}_{X_N}w\right\|_{\ell_{p}(X_{h})}}{\left\|w\right\|_{\ell_{p}(X_{h})}}.
\end{align}
Note that  $\M^{-1}_{X_N} = 
\PhiB_{X_{N}} \K^{-1}_{X_{N}} \in \R^{N \times N}$ exists
if and only if 
$
\K_{X_{N}} 
$ 
is invertible,
which shows that the inverse of the Kansa matrix is needed.

\subsection{
Stability for Helmholtz operators and positive definite SBFs}\label{inverse_stability}

Recall that  if the SBF $\varPhi_m$ has Hilbert-Schmidt expansion as in
(\ref{HS})
with coefficients $\widehat{\phi}(\ell)>0$ and 
$\widehat{\phi}(\ell)\sim |\nu_{\ell}|^{-m}$,
then its native space is $\Nn = H^m(\mathbb{S}^d)$.
For a Helmholtz type operator of order $\m = 2$ having the form 
$\opL =\alpha-\Delta $, with $\alpha>0$,
 the kernel generating the Kansa matrix 
 $\varPsi(x,y) = \opL^{(1)} \varPhi(x,y)$ is also an SBF 
with 
Hilbert-Schmidt expansion 
$\varPsi(x,y) =  \sum_{\ell=0}^{\infty}\sum_{m=-\ell}^{\ell}\widehat{\psi}(\ell)Y_{\ell,m}(x) Y_{\ell,m}(y)$ 
having 
coefficients 
$\widehat{\psi}(\ell)
= \widehat{\phi}(\ell)(\lambda_{\ell} +\alpha)>0$.

Because $\widehat{\psi}(\ell) \sim |\nu_{\ell}|^{1-m}$, the native space for $\varPsi$ is
$\mathcal{H}_\varPsi  = H^{m-1}(\mathbb{S}^2)$. 
In short, the Kansa matrix is simply the interpolation matrix for the SBF $\varPsi$: i.e.,
$\K_{X_N} = \bigl(\opL^{(1)} \varPhi(x_j,x_k)\bigr)  
= \bigl(\varPsi(x_j,x_k)\bigr) =\PsiB_{X_N}$.
In particular, it is symmetric positive definite, and therefore invertible. 

Stability bounds for the matrix $\K_{X_N}$ are a consequence of the following lemma, which uses 
decay properties of the needlets developed in \cite{NPW,NPW2}. The precise decay property we 
use here is \cite[Theorem 2.2]{NPW2}.
\begin{lemma} Let $\tau\in C^{\infty}([0,\infty))$ satisfy $\tau(t) =1$ for $0\le t\le 1$, $0\le \tau\le 1$ and $\mathrm{supp}(\tau)\subset [0,2]$.
The family of zonal kernels $\{T_M: M\in \N\}$ given by 
$$T_M(x,y) :=   \frac{8\pi}{(M+1)^2}\sum_{\ell=0}^{\infty} \sum_{|\mu| \le \ell} {\tau}(\frac{\ell}{M}) Y_{\mu}^{\ell}(x)Y_{\mu}^{\ell}(y)$$
has the following property: there is a constant $\Gamma$ so that
 for any $X\subset \Sph^2$, with separation radius $q>0$,  if $M \ge \Gamma/q$ then
that the  collocation matrix  $\mathbf{T_M}_X:=\bigl(T_M(x_j,x_k)\bigr)_{x_j,x_k\in X}$ has minimal eigenvalue greater than 1.
\end{lemma}
\begin{proof}
From \cite[Theorem 2.2]{NPW2}, we see that there is a constant $C_1$ so that
$$|T_M(x,y)| \le  \frac{C_1}{(1+M\dist(x,y))^3}$$ for all $x,y\in \Sph^2$.

Note that for $\ell\in \N$, the addition formula for spherical harmonics ensures that
$ \sum_{|\mu| \le \ell} |Y_{\mu}^{\ell}(x)|^2= \frac{2\ell+1}{4\pi}$, so $T_M(x,x)$ is constant in $x$,
and thus
$T_M(x,x)= \frac{8\pi}{(M+1)^2} \sum_{\ell=0}^{\infty} \tau(\frac{\ell}{M})  \frac{2\ell+1}{4\pi}\ge 2$
for every $x$.

For $x_j\in X$, we consider the sum of off-diagonal elements $\sum_{x_k\in X\setminus\{x_j\}} |T_M(x_j,x_k)|$. 
To this end, decompose $X=\bigcup_{n=1}^{\infty} E_n$, where $E_n =\{x_k\in X\mid nq\le \dist(x_k, x_j)\le (n+1)q\}$,
and note that $\#E_n\le 9 \pi^2n$. Thus 
$$\sum_{x_k\in X\setminus\{x_j\}}|T_M(x_j,x_k)|\le \sum_{n=1}^{\infty} \frac{C_1 9 \pi^2 n}{(1+Mnq)^3}
\le C_1 \frac{3\pi^4}{2}(Mq)^{-3}.
$$
Selecting $M$ so that $ C_1 \frac{3\pi^4}{2}(Mq)^{-3}<1$ ensures that for each $j$,
$$T_M(x_j,x_j) - \sum_{x_k\in X\setminus\{x_j\}}|T_M(x_j,x_k)|> \frac12 T_M(x_j,x_j)$$
and the lemma follows.
\end{proof}
The following lemma gives a lower bound for the numerical range of a conditionally positive definite collocation matrix (on the
subspace of admissible coefficients). Note that, by Remark \ref{SPD_is_CPD}, it applies also to (strictly) positive definite SBFs
with $\m=-1$.
\begin{lemma}\label{lower_numerical_range}
For a conditionally positive definite SBF $\varPhi$ of order $\tilde{m}$, having expansion (\ref{HS}) 
with coefficients $\widehat{\phi(\ell)}$ which obey the bounds 
$\gamma_1 |\nu_{\ell}|^{-m} \le \widehat{\phi(\ell)}$ 
for all $\ell>\tilde{m}$,
there is a constant $C$ so that  if $X\subset \Sph^2$ has separation distance $q$, then 
the $\lambda_{\min} = \min\left\{\frac{ c^T\PhiB_X c}{c^Tc}\mid c\in \R^N,  (\forall P\in \Pi_{\tilde{m}})\, \sum c_{x_j} P(x_j) = 0\right\}$
satisfies $\lambda_{\min} \ge C q^{2m-2}$.
\end{lemma}
\begin{proof}
Select $M>\min(\Gamma q,\tilde{m})$, and let $\beta_M:=\min_{\ell\le 2M} \widehat{\phi}(\ell)$. Since 
$|\nu_{\ell}| =\ell(\ell+1)$,
there is a constant
$C$ so that $\beta_M\ge C M^{-2m}$. 
Define $$S_M(x,y) :=
 \beta_M \sum_{\ell=0}^{\infty} \sum_{|\mu| \le \ell} {\tau}(\frac{\ell}{M}) Y_{\mu}^{\ell}(x)Y_{\mu}^{\ell}(y)
=\beta_M\frac{(M+1)^2}{8\pi} T_M(x,y).$$
When $\ell>\tilde{m}$,  the Fourier coefficients of $S_M$
are controlled by those of $\varPhi$:
$\widehat{s}(\ell) = \beta_M {\tau}(\frac{\ell}{M})\le \widehat{\phi}(\ell)$. It follows that
the collocation matrix 
$\bigl(\varPhi(x_j,x_k) - S_M(x_j,x_k)\bigr)_{x_j,x_k\in X}$
is conditionally positive definite of order $\tilde{m}$, and thus,
$\lambda_{\min}(\PhiB_X) \ge \lambda_{\min}(\mathbf{S_M}_X) \ge C M^2 \beta_M\ge C q^{2m-2}$.
\end{proof}
From this lemma, it follows that
$$\|\K_{X_N}^{-1}\|_{2\to 2}   = 
\| \PsiB_{X_N}^{-1}\|_{2\to 2} \le C q^{2 - 2(m - 1)}
= 
C q^{4 - 2m}
.$$
At the same time, the  interpolation matrix $\PhiB_{X_N}$ for the original kernel has 
$\ell_2\to \ell_2$ norm
$$ \| \PhiB_{X_N}\|_{2\to 2} \le N \|\varPhi\|_{\infty}  \le C q^{-2 }.$$
So in this case, 
\begin{equation}\label{full_inverse_stability}
\|\M_{X_{N}}^{-1}  \|_{2\to 2}
\le 
\|\K_{X_N}^{-1} \|_{2\to 2}   
\|\PhiB_{X_N} \|_{2\to 2}
\le C q^{2-2m}.
\end{equation}

This example can be modified as desired
to consider (not necessarily differential) operators 
which commute with $\Delta$ and therefore are diagonalized
by $(Y_{\ell}^{\mu})$.
In that case,
 $\|\M_{X_{N}}^{-1}  \|_{2\to 2}
\le 
 C q^{\m-2m}$, where $\m$ is rate of growth of the symbol of the operator.
\subsection{Stability in the conditionally positive definite case}
To prove stability, we will need to account for the smallest eigenvalue of $\K$, but also the contributions of $\opL \mathbf{P}$ and $\mathbf{P}$.
For this, we estimate on the stability of a basis for $\mathbf{P}$.

This  makes use of the 
Marcienkiewicz-Zygmund inequality  \cite[Theorem 4.2]{NPW},
which
states that there is a constant $\delta>0$ so that  for any $S\in \Pi_{L}$,
the estimate  $(1-\delta)\|S\|_1\le \sum_{\xi\in\Xi} |S(\xi)| w_{\xi} \le (1+\delta) \|S\|_1$
holds
with the weights $w_{\xi} = \mu(R_{\xi})$, obtained from the  Voronoi cell $R_{\xi}$. 
Of note, we have $B(\xi,q)\subset V_{\xi}\subset B(\xi,h)$, so there are constants $0<\tilde{c}_1\le \tilde{c}_2<\infty$ so that
$\tilde{c}_1q^2\le w_{\xi} \le \tilde{c_2} h^2$. 
A more-or-less direct application of this result  with $L=2\m$ guarantees $0<c_1\le c_2<\infty$ so that
for all $P \in \Pi_{\tilde{m}}$, 
\begin{equation}
\label{MZ}
c_1q^2 \sum_{\xi\in\Xi }|P(\xi)|^2 \le \int_{\Sph^2} |P(x)|^2 \dif x \le c_2 h^2 \sum_{\xi\in \Xi} |P(\xi)|^2 
\end{equation}
holds (since $S=|P|^2$ is a spherical harmonic of order $L=2\tilde{m}$).

\begin{lemma} \label{GRAM}
For any degree $\tilde{m} \in \N$, 
there exist constants $0<C_1\le C_2<\infty$ so that for any 
$L_2(\Sph^2)$-orthogonal  basis $\{Q_j\mid j\le N_{\tilde{m}}\}$ for $\Pi_{\tilde{m}}$ 
and any
finite set $X\subset \Sph^2$ with separation distance $q$ and fill distance $h$, the corresponding discrete Gram matrix
 $\mathbf{G}_X= (\mathbf{Q}^*\mathbf{Q}) = \left( \sum_{x_k\in X} Q_j(x_k) \overline{Q_{\ell}(x_k)}\right)_{j,\ell}$
 has spectrum which satisfies
 $$\sigma(\mathbf{G}_X) \subset \bigl[C_1 h^{-2} \min\|Q_j\|_2^2, C_2 q^{-2} \max\|Q_j\|_2^2\bigr].$$
\end{lemma}
\begin{proof}

 By orthogonality, the continuous Gram matrix $\int_{\Sph^2} Q_j(x) \overline{Q_{\ell}(x)}\dif x$ is diagonal.  
 The eigenvalues for the discrete Gram matrix $\mathbf{G}_X=(G_{j,\ell})_{j,\ell}$
can be determined from its numerical range. We consider  the induced quadratic form:
$$(a_j)_{j\le N_{\tilde{m}}} \mapsto 
\sum_{j,\ell=1}^{N_{\tilde{m}}}
 \overline{a_{\ell} }a_jG_{j,\ell}
= \sum_{x_k\in X} \left|\sum_{j=1}^{N_{\tilde{m}}} a_j Q_j(x_k)\right|^2$$

Using (\ref{MZ}), we have
$$
\sum_{x_k\in X} \left|\sum_{j=1}^{N_{\tilde{m}}} a_j Q_j(\xi)\right|^2 
\le 
\frac1{c_1}q^{-2} 
 \int_{\Sph^2} 
 \left|
    \sum_{j=1}^{N_{\tilde{m}}} a_j Q_j(x)
    \right|^2 
 \dif x = 
q^{-2}\frac{\left( \max \|Q_j\|_2^2\right)}{c_1} \sum_{j=1}^{N_{\tilde{m}}} |a_j|^2 $$
Similarly,
$$
\sum_{x_k\in X} \left|\sum_{j=1}^{N_{\tilde{m}}} a_j Q_j(\xi)\right|^2 
 \ge 
\frac1{c_2}h^{-2}  \int_{\Sph^2} \left|\sum_{j=1}^{N_{\tilde{m}}} a_j Q_j(x)\right|^2 \dif x = 
h^{-2}\frac{\left( \min \|Q_j\|_2^2\right)}{c_2}
\sum_{j=1}^{N_{\tilde{m}}} |a_j|^2, $$
and the lemma follows.
\end{proof}
 
\begin{proposition}\label{tps_stability}
	 Let $\opL$  be a Helmholtz type operator of order $\m = 2$ having the form $\opL =\alpha-\Delta $, with $\alpha>0$ 
	 and let $\varPhi_m$ be a conditionally positive definite SBF with expansion (\ref{HS}) and 
	 $\widehat{\phi}_{\ell} \sim \ell^{-2m} $ for all $\ell>\tilde{m}$.
	 Then $ \M_{X_N}$ is invertible and has stability bound $\|\M_{X_N}\|_{2\to 2}\le C q^{2-2m}$,
	 with $C$ a constant which depends on $\varPhi_m$
	 and mesh ratio $\rho$.
\end{proposition}
\begin{proof}
For the purposes of the proof, we let $(P_j)_{j\le N_{\tilde{m}}}$ be the standard (real) orthonormal spherical harmonic basis $(Y_{\ell}^{\mu})_{\ell\le \tilde{m},|\mu|\le \ell}$ given in (\ref{spherical_harmonic}).

As  in the strictly positive definite example above, we have
\begin{align*}
	\psi(x\cdot y) =\opL^{(1)} \phi(x\cdot y) 
	=\sum_{\ell=0}^{\infty} \overbrace{(\alpha+\lambda_{\ell})\widehat{\phi}_{\ell} }^{\widehat{\psi}_{\ell} } 
	\sum_{\mu=-\ell}^{\ell}
	Y^{\mu}_{\ell}(x)Y^{\mu}_{\ell}(y), 
\end{align*}
with  $\widehat{\psi}_{\ell} \sim \ell^{-2m}+ \ell^{-2m+2} \sim \ell^{-2(m-1)}$ for $ \ell > \tilde{m}$.
Hence, $\varPsi=\opL \varPhi$ is  conditionally positive definite of order $\tilde{m}$. 

Now, consider $z \in \R^{N}$ such that $ \M_{X_N} z=0$. Then,
$
	 [\K_{X_N} \vert \opL\mathbf{ P}]  [ A z \vert Bz ]^T=0.
$
Because  spherical harmonics of degree $\tilde{m}$ are invariant under $\opL$, and $A$ annihilates such spherical harmonics,
 we have $(\opL\mathbf{ P})^T A =\mathbf{0}_{N_{\tilde{m}}\times N}$.
It follows that
$$
\begin{pmatrix} 0 \\ 0 \end{pmatrix}= \begin{pmatrix} \K_{X_N}& \mathbf{ \opL P} \\ (\mathbf{ \opL P})^T & \mathbf{0} \end{pmatrix}   \begin{pmatrix} A z\\ B z\end{pmatrix}.
$$
By applying $(Az)^T$ to the first equation,
we conclude that
 $(Az)^T\K _{X_N}Az +(Az)^T(\opL \mathbf{ P} )Bz = 0$, which implies $(Az)^T\K _{X_N}Az =0$.
 Because   $\varPsi$ is conditionally positive definite of degree $\tilde{m}$, this
 implies 
$ A z=0$, which further implies that $ B z=0$.
Now, we use (\ref{lagrange_coeffs}), namely 
$z = \Id z = [\PhiB_{X_N} \vert  \mathbf{ P}] \left[ \begin{matrix} A\ B\end{matrix}\right] ^T z=0$,
to obtain $z=0$.
Hence,  $\M_{X_{N}}$
is injective and therefore invertible.

Given $z\in \R^{N}$, consider $y = \M_{X_N}^{-1} z$. We write 
 $$\begin{pmatrix} z\\0\end{pmatrix} =\begin{pmatrix}\K_{X_N}& \opL \mathbf{ P} \\ (\opL \mathbf{ P})^T & \mathbf{0}\end{pmatrix} \begin{pmatrix} {a}\\ {b}\end{pmatrix}
 \quad
 \text{and}
 \quad
 \begin{pmatrix}\PhiB_{X_N}&  \mathbf{ P} \\ \mathbf{ P}^T & \mathbf{0}\end{pmatrix} \begin{pmatrix} {a}\\ {b}\end{pmatrix} =\begin{pmatrix} {y}\\ 0\end{pmatrix}.$$ 
 To control $\|a\|_{\ell_2}$ and $\|b\|_{\ell_2}$ by $\|z\|_{\ell_2}$, we employ 
  \cite[Proposition 5.2]{FHNWW}, with $\vartheta =\lambda_{\min} \ge C q^{2m-4}$  determined by Lemma \ref{lower_numerical_range}.
A direct application of this result gives 
\begin{equation}\label{a-coeff}
\|a\|_{\ell_2} \le C q^{4-2m}\|z\|_{\ell_2(N)}.
\end{equation}
 With  Gram matrix
$\widetilde{\mathbf{G}}_{X_N} = (\opL \mathbf{ P})^T  \opL \mathbf{ P}=\bigl(\sum_{x_j\in X_N}\opL P_j (x_k) \opL P_{\ell} (x_j)\bigr)_{j,\ell}$,
the second estimate in \cite[Proposition 5.2]{FHNWW} ensures 
that 
 $$\|b\|_{\ell_2}\le C \left\|\bigl( \widetilde{\mathbf{G}}_{X_N}  \bigl)^{-1}\right\|_{2\to 2}^{1/2}  q^{2-2(m-1)} N \|z\|_{\ell_2(N)}
 \le 
 C \left( \frac{h}{c\sqrt{C_1}}\right)  q^{2-2(m-1)} N \|z\|_{\ell_2(N)}
 .$$ 
 The final inequality uses Lemma \ref{GRAM} applied to the $L_2$ orthogonal basis
   $(\opL P_j) =\bigl( (\alpha+\ell(\ell+1))Y_{\ell}^{\mu}\bigr)$, which shows that
 the discrete Gram matrix 
  has spectrum 
 $$\sigma\bigl(\widetilde{\mathbf{G}}_{X_N}\bigr)
 \subset
 [C_1 h^{-2} \alpha^2 ,C_2 q^{-2}  (\alpha+m(m-1))^2].$$
 It follows from quasi-uniformity and the fact that $N\le Cq^{-2}$, that 
 \begin{equation}\label{b-coeff}
 \|b\|_{\ell_2}\le C q^{3-2m}\|z\|_{\ell_2(N)}.
\end{equation} 
 
On the other hand,
a direct calculation shows that
$\|y\|_2 \le 
\|\PhiB_{X_N}\|_{2\to 2} \|a\|_{\ell_2(N)} +\| \mathbf{ P}^T \mathbf{ P}\|_{2\to 2}^{1/2} \| b\|_{\ell_2(N_{\tilde{m}})}$.
As in the positive definite case, $\|\PhiB_{X_N}\|_{2\to 2}\le N \|\varPhi\|_{\infty} \le C q^{-2}$,
so by (\ref{a-coeff}), $ \|\PhiB_{X_N}\|_{2\to 2} \|a\|_{\ell_2(N)} \le C q^{2-2m}\|z\|_{\ell_2}$.

Using Lemma \ref{GRAM}, this time with the orthonormal basis $(P_j) = (Y_{\ell}^{\mu})$, we can estimate
the spectrum of the discrete Gram matrix 
$\mathbf{G}_{X_N} = \bigl(\sum_{x_j\in X_N} P_j (x_k)  P_{\ell} (x_j)\bigr)_{j,\ell}$ with
$$\sigma\bigl(\mathbf{G}_{X_N}\bigr)
 \subset
 [C_1 h^{-2} ,C_2 q^{-2} ].$$
Thus, by (\ref{b-coeff}), we have
$\| \mathbf{ P}^T \mathbf{ P}\|_{2\to 2}^{1/2}\| b\|_{\ell_2(N_{\tilde{m}})}
\le 
\left(\sqrt{C_2} q^{-1}\right) Cq^{3-2m}  \|z\|_{\ell_2(N)} 
\le
C q^{2-2m}$.
Therefore, it follows that
$\|y\|_2\le C q^{2-2m}   \|z\|_{2}$
which implies
$\|\M_{X_N}^{-1}\|_{2\to 2} \le C q^{2-2m}$ as desired.
 \end{proof}

\section{Localization or restricted thin plate spline FD matrices}\label{S:Localization}

When using the full RBF-FD matrix, a basic problem is the construction of $\M_{X_N}$ -- 
this is roughly equivalent to solving a large interpolation problem (solving an $N\times N$ system), 
followed by  a large  ($N\times N$) matrix multiplication. 
To this end, it is desirable to consider the problem of solving a number of small systems 
-- in other words, we consider instead using one  
``stencil'' for each point in $\Xi$, so that each stencil uses at most $M\ll N$ nearby points.

In order to have nice theoretical bounds from \cite{HNW-p}, we will restrict the localization to restricted thin-plate splines
described in section \ref{SSS:RTPS}.

\subsection{Local Lagrange functions and local stencils}
We now present an alternative local stencil FD method which permits some 
theoretical error estimates.
In \cite{HNW-p} it is shown 
that the Lagrange functions
\begin{equation}\label{full_lagrange}
\chi_{x_j} = \sum_{k=1}^N A_{j,k}\varPhi(\cdot, x_k)
 +\sum_{\ell=0}^{\tilde{m}}\sum_{\mu=-\ell}^{\ell} B_{j,\ell,\mu} Y_{\ell}^\mu
 \end{equation}
  in the restricted thin-plate spline setting enjoy a number of 
  analytic properties. Of interest particular interest here is that 
  the  coefficients determined by (\ref{lagrange_coeffs}), satisfy
\begin{equation}\label{lagrange_props}
|A_{j,k} |\le C h^{2-2m} \exp{\left(-\nu
    \frac{\dist(x_j,x_k)}{h}\right)}
.
\end{equation}
with  $\nu>0$; 
the constant $C$ depends on the mesh ratio $\rho = h/q$ for $X_N$, 
while $\nu$ depends only on the kernel.

\subsubsection{Local Lagrange functions}
In \cite{FHNWW}, it is shown that the space $S(X_N)=S(X_N,m-1)$ possesses an easily computed
stable basis $(b_j)_{j=1}^N$
consisting of ``local Lagrange functions''. 
 Since then, this construction 
(along with its  desirable properties) has been demonstrated for
kernels on other domains \cite{HNRW2} 
and on bounded regions in $\R^d$ (\cite{HNRW1}).

The coefficients of the basis function
$b_{j}\in 
S(X_N)
$ are obtained by solving  a relatively (with respect to $N$) 
smaller system: 
namely, take 
$\Upsilon(x_j):= \{x_k\in X_N\mid \mathrm{dist}(x_j,x_k) \le K h|\log h|\},$
where $K$ is fixed constant depending on $m$; then there are functions of the form
\begin{equation}\label{local_lagrange}
b_j = \sum_{x_k\in \Upsilon(x_j)} \check A_{j,k}\varPhi(\cdot, x_k)
 +\sum_{\ell=0}^{\tilde{m}}\sum_{\mu=-\ell}^{\ell} \check B_{j,\ell,\mu} Y_{\ell}^\mu
 \end{equation}
 satisfying for all  $x_k\in \Upsilon(x_j)$,  $b_j(x_k) = \delta_{j,k}$. 
 Clearly 
 $b_j\in 
 S(\Upsilon(x_j))
 $, and 
 the coefficients $\check A_{j,k}, \check B_{j,\ell,\mu}$ are obtained by solving a linear system
 of size $(\#\Upsilon(x_j) + N_{m-1})$. 
 Because 
 $
 S(\Upsilon(x_j)),
 \subset 
 S(X_N)
 $, we may consider $b_j$ as having coefficients
 $(\check A_{j,k})_{k=1}^N$, with $\check A_{j,k}=0$ when 
 $x_k\notin \Upsilon(x_j)$; i.e., we extend by zero
 from $\Upsilon(x_j)$ to $X_N$.
 
 The local Lagrange coefficients $\check A_{j,k},  \check B_{j,\ell,\mu}$ of $b_j$ 
 differ from the coefficients $A_{j,k},  B_{j,\ell,\mu}$ of $\chi_{x_j}\in 
  S(X_N)
 $ 
 given in (\ref{full_lagrange}) very slightly, as the following lemma shows.
 
\begin{lemma}\label{coeff_distance}
For 
$\chi_{x_j}\in  
S(X_N)
$ given in (\ref{full_lagrange}) and 
$b_j\in 
S(\Upsilon(x_j))
\subset 
S(X_N)
$
as in (\ref{local_lagrange}),
the coefficients satisfy the bounds
\begin{equation}\label{coeff_dist_inequality}
\sum_{k=1}^N | A_{j,x_k}-\check A_{j,x_k}|\le C h^{K\nu-\kappa_m}, 
\qquad
\qquad
\sum_{\ell=0}^{\tilde{m}}\sum_{\mu=-\ell}^{\ell}|B_{j,\ell,\mu}- \check B_{j,\ell,\mu}| \le Ch^{K\nu-\kappa_m}
\end{equation}
\end{lemma}
This lemma follows from \cite[Section 6]{FHNWW}
as well as \cite[Section 4]{HNRW1}.

 Because the coefficients of $b_j$ are close to those of $\chi_{x_j}$, the two functions are
close in any normed space $Z$ in which $\max_{x\in \Sph^2}\|\varPhi(\cdot,x)\|_Z$ 
and $\max_{\ell,\mu}\|Y_{\ell}^{\mu}\|_Z$
are finite.
This is a simple consequence of the triangle inequality:
\begin{equation}\label{generic_norm_bound}
\|b_j - \chi_{x_j}\|_Z 
\le
\left( \sum_{k=1}^N |A_{j,k} - \tilde{A}_{j,k}| \right)\max_{k}\|\varPhi(\cdot, x_k)\|_Z
+ 
\left( \sum_{\ell=0}^{m-1}\sum_{|\mu|\le \ell} |B_{j,\ell,\mu} - \tilde{B}_{j,\ell,\mu}| \right)
\max_{\ell,\mu}\|Y_{\ell}^{\mu}\|_Z
\end{equation}
and Lemma \ref{coeff_distance}.
\begin{lemma}\label{difference}
For a differential operator $\opL$ of order $\m< m-d/2$, 
the functions $\bigl(\opL b_{j} \bigr)_{j=1}^N$ satisfy
$$\bigl\| \opL b_{j}  - \opL \chi_{x_j}\bigr\|_{\infty}\le Ch^{J}$$
with $J= K\nu -\kappa_m$.
\end{lemma}
\begin{proof}
We note that $\Nn$ is embedded in $C^{{\m}}(\Sph^2)$, 
and 
for any $y,x\in \Sph^2$,  
$\|\varPhi(\cdot,x)\|_{C^{\m}}
\le 
C \|\varPhi(\cdot,y)\|_{\Nn}$. Thus,  we have 
$$
\|\opL b_{j}- \opL \chi_{x_j}\|_{L_\infty(\Sph^2)} 
\le 
C\| b_{j}- \chi_{x_j}\|_{C^{{\m}}({\Sph^2})} 
\le C h^{J} $$
by applying  Lemma \ref{coeff_distance} in conjunction with 
(\ref{generic_norm_bound}) with $Z= C^{\m}(\Sph^2)$.
\end{proof}

Although the matrix $\M_{X_N}^{\sharp}$ is not sparse, it does have
 rapid off-diagonal decay. This is demonstrated in the following lemma.
\begin{lemma}
For a differential operator $\opL$ of order $\m< m-d/2$, 
the matrix $\bigl(\opL b_{j}(x_k)\bigr)_{j,k=1\dots N}$ satisfies
$$\bigl| \opL b_{j}(x_k)\bigr|\le C h^{-{\m}} \left(1+\frac{\dist(x_j,x_k)}{h}\right)^{-J}$$
with $J= K\nu -\kappa_m$.
\end{lemma}
\begin{proof}
We have   for $|x-x_j|\le R$ that
$|\opL \chi_{x_j}(x)|\le 
C\|\chi_{x_j}\|_{C^{\m}\bigl(\Sph^2\setminus B(x_j,R)\bigr)}$.
By applying the zeros estimate 
\cite[Theorem A.11]{HNW-p}
to $\chi_{x_j}$,
this ensures that
$|\opL \chi_{x_j}(x)|\le 
Ch^{m-{\m}-d/2} \|\chi_{x_j}\|_{H^m\bigl(\Sph^2\setminus B(\xi,R)\bigr)}$.
The Sobolev norm of $ \chi_{x_j}$
can be estimated by \cite[Lemma 5.4]{HNW-p} as
$
\|\chi_{x_j}\|_{H^m(\Sph^2\setminus B(\xi,R))}
\le 
C\rho^{d/2-m} h^{d/2-m}e^{-\nu \frac{R}{h}}.
$
Thus,
$$|\opL \chi_{x_j}(x)|\le C h^{-\m}e^{-\nu \frac{|x-x_j|}{h}}$$
holds.
Finally, an application of Lemma \ref{difference} gives the result.
\end{proof}

\subsubsection{Local stencil version of the RBF-FD matrix}
Given the local Lagrange basis, we consider a  different ``local stencil'' variant of
the FD matrix, namely
$$\M_{X_N}^{\sharp}:= \bigl( \opL b_{j}(x_k)\bigr)_{j,k} = [\K_{X_N} \vert \opL P] \left[ \begin{matrix} A\\  B\end{matrix}\right]$$
where $A = (A_{j,k})$ is the sparse matrix of kernel coefficients and $B =(B_{j,\ell,\mu})$
is the $N_{m-1}\times \#\Upsilon(x_j)$ matrix of spherical harmonic coefficients.
Unlike the more conventional small stencil construction 
$\M_{X_N}^{\circ}$ 
described in the introduction,
this matrix is not row-sparse: the 
functions $b_j$ only have prescribed zeros in $\Upsilon(x_j)$. 
However, since $\#\Upsilon(x_j) \sim (\log N)^2$,
the system used to generate $b_j$ has size 
$\mathcal{O}\bigl((\log N)^2\bigr)\ll N$.

The local stencil is a small perturbation of the full stencil -- we show this by
controlling $\| \M_{X_N}^{\sharp}-\M_{X_N}\|_{p\to p}$ by row and column sums.
In other words, by
applying Lemma \ref{difference} we have
\begin{equation}\label{row_col}
\begin{rcases*}
\| \M_{X_N}^{\sharp}-\M_{X_N}\|_{1\to 1} \ \  \, 
=\max_k \sum_{j=1}^N |\opL b_{j}(x_k) - \opL \chi_{x_j}(x_k)|\\
\| \M_{X_N}^{\sharp}-\M_{X_N}\|_{\infty\to \infty}
=\max_j \sum_{k=1}^N |\opL b_{j}(x_k) - \opL \chi_{x_j}(x_k)|
\end{rcases*}
\le 
 \sum_{j=1}^N Ch^J\le Ch^{J-2}
\end{equation}
since $N\le C h^{-2}$. The bound  $\| \M_{X_N}^{\sharp}-\M_{X_N}\|_{p\to p}\le  Ch^{J-2}$
then follows for all $p\in[1,\infty]$.

An advantage of this setup is that we are able to retain the global consistency rates, thanks
to the nearness of the local Lagrange functions to the global Lagrange functions. 
First, we can give a consistency estimate
\begin{theorem}\label{local_consistency}
For $K\nu\ge \kappa_m+2m-1-\m$, we have
$\| \M_{X_N}^{\sharp} (f|_{X_N}) - (\opL f)|_{X_N}\|_{\ell_{\infty}} \le C h^{2m -1- \m}\|u\|_{H^{2m}}$.
\end{theorem}
\begin{proof}
By the consistency result (\ref{double_consistency}), we have 
$\| \M_{X_N} (f|_{X_N}) - (\opL f)|_{X_N}\|_{\ell_{\infty}} \le C h^{2m - 1-\m}\|u\|_{H^{2m}}$
for the full FD matrix. Thus if $K\nu -\kappa_m=J\ge 2m - 1-\m$, the result follows.
\end{proof}

Finally, although there is no satisfactory stability theory for general differential operators, we know at least that local
stencil provides inverse stability commensurate with the full problem. 
\begin{theorem}\label{near_stability}
Suppose $ K\nu >2+\kappa_m$. For any $p\in[1,\infty]$, if  $X_N\subset \Sph^2$ with $h$ sufficiently small, we have
$$\| (\M_{X_N}^{\sharp} )^{-1}\|_{p\to p} \le 2\| (\M_{X_N} )^{-1}\|_{p\to p}. $$
\end{theorem}
\begin{proof}
For any submultiplicative norm, if $\|\M_{X_N}^{\sharp}-\M_{X_N}\|<\frac1{2\|(\M_{X_N})^{-1}\|}$ then a standard Neumann
series argument gives
$$\bigl\|\bigl(\M_{X_N}^{\sharp}\bigr)^{-1}\bigr\| \le \bigl\|\bigl(\M_{X_N}\bigr)^{-1}\bigr\|
\,\sum_{j=0}^{\infty} \left\| \bigl(\Id- \M_{X_N}^{\sharp}(\M_{X_N})^{-1}\bigr)\right\|^j
\le 2 \|(\M_{X_N})^{-1}\|.
$$
Since $\M_{X_N}^{\sharp}-\M_{X_N} =  \bigl(\opL b_{j}(x_k) - \opL \chi_{x_j}(x_k)\bigr)_{j,k}$,
$\|\M_{X_N}^{\sharp}-\M_{X_N}\|\to 0$ as $h\to 0$.
\end{proof}

Taken together, Theorem \ref{near_stability} and Theorem \ref{local_consistency} show that if the
full FD matrices are stable  
(i.e., if  for some constant $C$, $\|\M_{X_N}^{-1}\|_{\infty\to\infty} <C$ for all $X_N$),
then the local version is convergent:
$$\left\|u|_{X_N} - (\M_{X_N}^{\sharp})^{-1} (f|_{X_N})\right\|_{\ell_{\infty}}\le 2\|(\M_{X_N})^{-1}\|_{\infty\to \infty} \times  C h^{2m -1- \m}\|u\|_{H^{2m}}.
$$

\section{Error analysis for the localized RBF-FD approach}\label{S:error}
In this section, we present an error analysis for the solution of Helmholtz-type differential equations 
using the RBF-FD method  with a full FD matrix  as well as the  with localized kernels. 
This provides a problem where we can guarantee stability of the full (and therefore localized) RBF-FD matrices.
\begin{theorem}\label{full_theorem}
	 Let $\opL$  be a Helmholtz type operator of order $\m = 2$ having the form $\opL =\alpha-\Delta $, 
	 with $\alpha>0$ and let $\Phi$ be the restricted thin plate spline of order $m$.
	 We consider the differential equation $\opL u = f$ and its discretization via the linear system
$	 
	 \M_{X_N}c = \Sigma_{N}f
$
	and the resulting approximation 
	${u}_{N;f}:=\sum_{j=1}^{N} c_j\chi_{x_j} $.
	Then, there is a constant $C$ such that
	\begin{align*}
		\left\|u- {u}_{N;f}\right\|_{L_{2}} \le C h^{m-1} \left\|f \right\|_{W_{2}^{m-1}(\Sph^2)}  
		 \text{ for all } h\le h_0.
	\end{align*}
\end{theorem}
\begin{proof}
By  positivity of $-\Delta$, 
we have
$\alpha\| f\|_{L_{2}} \le   \| \opL f\|_{L_{2}} $,
and so
\begin{align*}
	\left\|u- u_{N;f}\right\|_{L_{2}}\le 
	\frac1\alpha\left\| \opL \left( u- u_{N;f} \right)\right\|_{L_{2}}= \frac1\alpha\left\| f- \opL u_{N;f}\right\|_{L_{2}}.
\end{align*}
From 
$
	\opL u_{N;f}(x_k)=\sum_{n=1}^{N} c_{n} \opL\chi_{\xi_{n}} (x_k)=f(x_k),
$
we can deduce that $\opL u_{N;f}$ is the kernel-based interpolant using the kernel $\Psi=\opL \Phi$. 
As outlined in \cite{NRW},
 we can modify the conditionally positive definite kernel interpolant in the first modes to get an interpolant based on a 
 strictly positive definite kernel. Thus, we will get
 \begin{align}
 	\left\| f- \opL u_{N;f}\right\|_{L_{2}}\le C h^{m-1} \left\| f\right\|_{W_{2}^{m-1}(\Sph^2)},
 \end{align}
 where the exponent $m-1$ is due to the change $\Phi \to \Psi$.
 \end{proof}
 
 A similar convergence result can now be shown for the localized SBF-FD method using matrices $ \M_{X_N}^{\sharp}$.

\begin{theorem}
	 Let  $\opL =\alpha-\Delta $, 
	 with $\alpha>0$ and let $\Phi$ be a restricted thin plate spline of order $m$.
	 We consider the differential equation $\opL u = f$ and its discretization via the linear system
$	 \M_{X_N}^{\sharp}c = \Sigma_{N}f$
	and the resulting approximation 
	$\check{u}_{N;f}:=\sum_{h=1}^{N} \check{c}_{j} b_{j}$.
	Then, there is a constant $C$ such that
	\begin{align*}
		\left\|u- \check{u}_{N;f}\right\|_{L_{2}} \le 
		C\left( h^{m-1} + 
		h^{J+2-4m}\right) \left\|f \right\|_{W_{2}^{m-1}(\Sph^2)}\quad \text{for all } h\le h_0.
	\end{align*}
\end{theorem}
This motivates us to choose $J=5m-3$.
\begin{proof}
By the triangle inequality, we have
$
	\left\|u- \check{u}_{N;f}\right\|_{L_{2}} \le\left\|u- u_{N;f}\right\|_{L_{2}}+ \left\| \check{u}_{N;f}- u_{N;f}\right\|_{L_{2}}
	$,
	where we use the ``full'' FD approximant $u_{N;f}:=\sum_{j=1}^{N} c_{j} \chi_{x_j} $
	obtained in Theorem (\ref{full_theorem}).
We may split the error $\check{u}_{N;f}- u_{N;f}$ as
\begin{align*}
	\check{u}_{N;f}- u_{N;f}
	=
	\sum_{j=1}^{N} \left(c_{j}-\check{c}_{j} \right) \chi_{x_{j}} 
	+ 
	\sum_{j=1}^{N} \check{c}_{j} \left( \chi_{x_{k}}-b_{j}\right). 
\end{align*}
For the first term in this splitting, we can employ the Riesz basis property (\cite[Theorem 5.3]{FHNWW}) to obtain
	\begin{align*}
		\left\| \sum_{j=1}^{N} \left(c_{j}-\check{c}_{j} \right) \chi_{x_{j}} \right\|^2_{L_{2}}  
		\sim 
		h^{2} \sum_{j=1}^{N} \left(c_{j}-\check{c}_{j} \right)^2 
\end{align*}
Using the fact that $\check{c} = (\M^{\sharp}_{X_N})^{-1}  \Sigma_{N}f= (\M^{\sharp}_{X_N})^{-1} \M_{X_N}c$, we calculate
	\begin{align*}
		 \check{c}-c 
		 =(\M^{\sharp}_{X_N})^{-1} \M_{X_N}c  -c=(\M^{\sharp}_{X_N})^{-1}\left(\M_{X_{N}}- \M^{\sharp}_{X_N} \right)c.
	\end{align*}
	Moreover, we have 
	\begin{align*}
		\left(\M^{\sharp}_{X_N}\right)^{-1}=\left(\M_{X_{N}}\left( \Id +\M^{-1}_{X_{N}}\left(\M^{\sharp}_{X_{N}} -\M_{X_{N}} \right) \right)  \right)^{-1}= 
		\left( \Id +\M^{-1}_{X_{N}}\left(\M^{\sharp}_{X_{N}} -\M_{X_{N}} \right) \right)^{-1}\M^{-1}_{X_{N}}.
	\end{align*}
	Hence, we obtain 
	\begin{align*}
		\check{c}-c=\left( \Id +\M^{-1}_{X_{N}}\left(\M^{\sharp}_{X_{N}} -\M_{X_{N}} \right) \right)^{-1}\M^{-1}_{X_{N}}\left(\M_{X_{N}}- \M^{\sharp}_{X_N} \right)c.
	\end{align*}

Now use %
$\| \M_{X_N}- \M_{X_N}^{\sharp}\|_{2\to 2} \le C h^{J- 2 }$,
which is a consequence of
(\ref{row_col}), 
 with $F:=\M_{X_N}^{-1} \left( \M_{X_N} -  \M_{X_N}^{\sharp}\right)$, 
 to deduce that
	$\|F\|_{2\to 2}
	\le C h^{J-2}  \| (\M_{X_N})^{-1}\|_{2\to 2}  <1$
	by Proposition \ref{tps_stability}.

Thus by standard estimates, we obtain 
\begin{align*}
	\left\| c - \check{c}\right\|_{2} 
	\le 
	\left( \frac{\|F\|_{2\to 2}}{1-\|F\|_{2\to 2}}\right) \left\| c \right\|_{2}
	\lesssim  
	\left( \frac{\|F\|_{2\to 2}}{1-\|F\|_{2\to 2}}\right)  \|  \M_{X_N}^{-1}\|_{2\to 2}  \left\| \Sigma_N f\right\|_{2} 
	\lesssim h^{J-2}  \|  \M_{X_N}^{-1}\|_{2\to 2} ^2 \left\| \Sigma_N f\right\|_{2} .
\end{align*}
By the Riesz basis property (\cite[Theorem 5.3]{FHNWW}), followed by \eqref{full_inverse_stability},  we obtain the estimate 
	\begin{align*}
		\left\| \sum_{j=1}^{N} \left(c_{j}-\check{c}_{j} \right) \chi_{x_j} \right\|_{L_{2}} \le  
		C h^{J -1  }  \|  \M_{X_N}^{-1}\|_{2\to 2} ^2 \left\| \Sigma_N f\right\|_{2} 
		\le C h^{J+3-4m} \left\| \Sigma_N f\right\|_{2}  .
	\end{align*}
For the term $\sum_{j=1}^{N} \check{c}_{j} \left( \chi_{x_{j}}-b_{j}\right)$, 
we begin by  simply expanding:
	\begin{align*}
		\left\|\sum_{n=1}^{N} \check{c}_{n} \left( \chi_{x_{n}}-b_{n}\right) \right\|^2_{L_{2}}
		=\left\|  \sum_{n,\tilde{n}=1}^{N} \check{c}_{n}  \check{c}_{\tilde{n}}  (\chi_{x_n} -b_n)(\chi_{x_{\tilde{n}}} -b_{\tilde{n}}) \right\|_{L_{1}}
	 	\le 
	 	C\max_{1\le j \le N} \|\chi_{x_j} -b_j\|^2_{\infty}  \sum_{n,\tilde{n}=1}^{N} \check{c}_{n}  \check{c}_{\tilde{n}} .
	\end{align*}
The final double sum is simply $ \|\check{c} \|^{2}_{2}$, which can be estimated as
 $ \|\check{c} \|^{2}_{2}\le
 \left\|(\M_{X_N}^{\sharp})^{-1}\right\|_{2\to 2}^2 \left\| \Sigma_N f\right\|_{2}^2 $.
This gives
\begin{align*}
\left\|\sum_{n=1}^{N} \check{c}_{n} \left( \chi_{x_{n}}-b_{n}\right) \right\|^2_{L_{2}}
	 &\le  
	 C\max_{1\le j \le N} \|\chi_{x_j} -b_j\|^2_{\infty}
	   \left\|( \M_{X_N}^{\sharp})^{-1} \right\|^{2}_{2 \to 2} 
	   \left\| \Sigma_{N}f \right\|^{2}_{2}
	 \le  C
	 h^{2J} 
	   \left\| \M_{X_N}^{-1} \right\|^{2}_{2 \to 2} 
	   \left\| \Sigma_{N}f \right\|^{2}_{2}
\end{align*}
by using Theorem \ref{near_stability}  and  Lemma \ref{difference} (namely $ \|\chi_{x_j} - b_j\|_{\infty}\le C h^{J}$) in the final inequality.
From this, we have
$
\left\|\sum_{n=1}^{N} \check{c}_{n} \left( \chi_{x_{n}}-b_{n}\right) \right\|_{L_{2}}
\le
C
	 h^{J+2-2m} 
	   \left\| \Sigma_{N}f \right\|_{2}$
	   by applying Proposition \ref{tps_stability}.
	   	Thus, we obtain
	   	\begin{align*}
	   		\left\|u- \check{u}_{N;f}\right\|_{L_{2}} \le 
			C\left( h^{m-1} \left\|f \right\|_{W_{2}^{m-1}(\Sph^2)} + 
			h^{J+3-4m} \left\| \Sigma_{N}f \right\|_{\ell_2(X_N)} \right).
	   	\end{align*}
	   	By using $\left\| \Sigma_{N}f \right\|_{\ell_2(X_N)}  \le \sqrt{N} \left\|f \right\|_{W_{2}^{m-1}(\Sph^2)} \le h^{-1}\left\|f \right\|_{W_{2}^{m-1}(\Sph^2)}$, the result follows.
\end{proof}
We point out, that the solution can be obtained by solving an almost sparse linear system (non-zero entries of order $N\log(N)$). The error induced by using local 
Lagrange functions and the global Lagrange functions will decrease with increasing $K$ and hence growing band-width.
Moreover, we point out that our estimates are technically much easier to derive than the error estimates \cite{DAVYDOV2019} as our analysis does not rely on charts.

\appendix
\section{Kernels on the sphere
}\label{S:kernel_examples}

In this section, 
we consider 
 SBFs of the form $\varPhi(x,y) = \phi(x\cdot y)$. 
%
%
Many SBFs are the restriction to 
$\mathbb{S}^2\times \mathbb{S}^2$ 
from radial, translation invariant kernels (RBFs) on $\R^3$: i.e., restrictions of kernels of the form 
$\varPsi(x,y) = \psi(|x-y|)$.
By writing $ |x-y| =  \sqrt{2-2t} $,
we may write the restricted RBF  in terms of a single real variable $t = x\cdot y$,
so 
\begin{equation} \label{first_restriction_relation}
\phi(t) = \psi(\sqrt{2-2t}).\end{equation}

The restricted polyharmonic splines are a well-known class of conditionally positive SBFs given by
$$\phi(t)= \left\{ \begin{array}{ll }(1-t)^{s} \log (1- t), & s \in \mathbb{N}, \\ 
                                     (1-t)^{s}, & s \in \mathbb{N} - \frac{1}{2}, \end{array} \right. $$
for all $t\in[-1,1]$. In terms of  $R(t) := \sqrt{1-t}$ (which is invertible on $[-1,1]$),
we can express $\phi(t)$ as
$$\phi(t)= \left\{ \begin{array}{ll }R(t)^{2s} \log R(t), & s \in \mathbb{N}, \\ 
                                     R(t)^{2s}, & s \in \mathbb{N} - \frac{1}{2}. \end{array} \right. $$
In this appendix, we focus on strictly positive definite SBFs.
In contrast to the restricted polyharmonic splines, which are discussed throughout the body of the article,
 such kernels require tuning of a shape parameter.

\subsection{Strictly positive definite kernels and shape parameters}\label{SS:shape}
By rescaling the RBF by $\sigma>0$ as
$\varPsi_{\sigma}(x,y) = \psi(\frac{|x-y|}{\sigma})$, 
we obtain a new positive definite 
radial basis function.\footnote{Briefly, by  scaling $k\to k(\cdot/\sigma)$, 
we rescale the Fourier transform, 
$\widehat{k} \to \sigma^d \widehat{k}(\sigma \cdot)$. The positive definiteness of $k(\cdot/\sigma)$
is a consequence of the positivity of its Fourier transform, which is unchanged by rescaling.}
This means that the RBF has a parameter 
(roughly the reciprocal of  the ``shape parameter'' $\epsilon$ defined below) which must me be set.
By changing this parameter, one alters the scaling of the RBF; 
this is desirable (indeed, necessary) for many
practical problems: e.g., for a finite point set $X_N\subset \R^d$, 
the collocation matrix 
$$
\mathbf{\Psi}_{X_N} =\bigl(\varPsi_{\sigma}(x_j,x_k)\bigr)_{j,k}
$$ 
will have a large condition number if $\sigma$ is large relative to the 
nearest neighbor distance $q= \min_{j\neq k} |x_j-x_k|$. 
Thus, in order to treat (e.g., by interpolation, quadrature, etc.)  densely sampled data, one 
 may wish to choose a small value of $\sigma$ to stabilize the problem. 
 This is also the case when using the
 SBF 
 obtained by restricting 
 the RBF $\varPsi_{\sigma}$.
 
 Choosing small $\sigma$ may come at a cost, however:
the shape parameter affects  the native space, both 
by modifying the inner product, and in some cases by altering the underlying set.
The known convergence results for RBF approximation and interpolation are 
most often  assume a  fixed shape parameter. 
Indeed, for most RBFs (which have a positive, continuous Fourier transform and therefore lack a 
{\em Strang-Fix condition} \cite{BR}), choosing $\sigma \propto h$, 
will lead to non-convergence \cite{Powell}. 
To date there are few direct approximation results which deal 
with multi-scale RBF approximation \cite{FloaterIske,NW_multi, HR,WendSloan}.

In what follows, we'll use the notation $R(t) := \sqrt{1-t}$ (which is invertible on $[-1,1]$)
and the modified ``shape parameter'' $\epsilon := \frac{\sqrt{2}}{\sigma}$,
which yields the relation, via (\ref{first_restriction_relation}),
$\phi(t) = \psi(\epsilon R(t))$.

Let us now give a few examples of zonal kernels, along with their native spaces.

\renewcommand{\arraystretch}{2}
\begin{table}
\centering
\begin{tabular}{|c |c |c|}
\hline 
SBF& 
$\phi$
&
Native space
\\
\hline
\begin{tabular}{@{}c@{}}Wendland 2.5 \\\end{tabular}
&
$\left(1-\epsilon R\right)_+^4 \, \bigl(1+4(\epsilon R) \bigr)$
&
 $ H^{2.5}(\Sph^2)$
\\
\hline
\begin{tabular}{@{}c@{}}Wendland 3.5 \\ \end{tabular}
&
$\left(1-\epsilon R\right)_+^6 \, \bigl(3+18(\epsilon R) + 35\bigl(\epsilon R\bigr)^2\bigr)$
&
$ H^{3.5}(\Sph^2)$
\\
\hline
\begin{tabular}{@{}c@{}}Wendland 4.5 \\ \end{tabular}
&
$
\left(1-\epsilon R\right)_+^8 \, 
\bigl(1+8(\epsilon R) + 25\bigl(\epsilon R\bigr)^2 + 32(\epsilon R)^3\bigr)
$
&
$ H^{4.5}(\Sph^2)$
\\
\hline
\begin{tabular}{@{}c@{}}Mat{\' e}rn 2.5\\  \end{tabular}
&
$
e^{-\epsilon R} \,\bigl( 1+ \epsilon R \bigr)
$
&
$H^{2.5}(\Sph^2)$
\\
\hline
\begin{tabular}{@{}c@{}}Mat{\' e}rn 3.5\\  \end{tabular}
&
$
 e^{-\epsilon R} \,\bigl(3 + 3(\epsilon R ) + (\epsilon R)^2 \bigr)
$
&
$H^{3.5}(\Sph^2)$
\\
\hline
\begin{tabular}{@{}c@{}}Mat{\' e}rn 4.5\\  \end{tabular}
&
$
 e^{-\epsilon R(t)} \,\bigl(15 + 15(\epsilon R ) + 6(\epsilon R)^2+ (\epsilon R)^3 \bigr)
$
&
$H^{4.5}(\Sph^2)$
\\
\hline
Gaussian& 
$e^{- \bigl(\epsilon R\bigr)^2}$
&
$\mathcal{N}(\phi)\subset C^{\infty}(\Sph^2)$
\\
\hline
Inverse multiquadric
&
$
{\bigl(1+\bigl(\epsilon R\bigr)^2\bigr)^{-\beta}}
$
&
$\mathcal{N}(\phi)\subset C^{\infty}(\Sph^2)$
\\
\hline
\end{tabular}
\caption{Some SBFs and their native spaces.}
\end{table}

\section{Differential equations on 
the sphere:
assembling the FD  matrix}\label{S:Practical}
In this section, we consider how to assemble the kernel FD matrix.
 Generally, this requires some understanding of 
the expression of $\opL$ in spherical coordinates (a different method  to construct $\M_{X_N}$
in Cartesian coordinates has been given in \cite{FW})
and how to use it to obtain an expression for $\opL^{(1)} \varPhi$.

The basic challenge  is to assemble the Kansa type matrix
$\K_{X_N} = \bigl(\opL^{(1)} \varPhi(x_j,x_k)\bigr)_{j,k}$ for a few first and second
order linear differential operators $\opL$.

\subsection{Working in coordinates}
In what follows, we use spherical coordinates $(\theta,\varphi)\in [0,2\pi)\times [0,\pi]$
to describe points on the sphere.
This involves  the convention $x =\cos \theta \sin \varphi$, $y= \sin \theta \sin \varphi$
and $z=\cos \varphi$.

The surface gradient is
$\nabla f(\theta, \varphi) 
=
 \frac{\partial f}{\partial \varphi} \vec{\varphi} 
 + 
 \frac{1}{\sin^2 \varphi}\frac{\partial f}{\partial \theta} \vec{\theta}$. 
 Here $\vec{\varphi}$ and $\vec{\theta}$ are the basic tangent vectors for the spherical coordinate
 system:
 $$ 
 \vec{\varphi}  
 = 
 \begin{pmatrix}\cos \theta \cos \varphi\\ \sin \theta \cos \varphi\\ - \sin \varphi\end{pmatrix}
 \quad
 \text{ and }
 \quad
 \vec{\theta} = \begin{pmatrix}-\sin \theta \sin \varphi\\ \cos \theta \sin \varphi\\ 0\end{pmatrix}.$$

The spherical divergence operator applied to a vector field 
$F = F_{\varphi} \vec{\varphi}+F_{\theta} \vec{\theta}$
gives 
\begin{equation}\label{div}
\mathrm{div} F = 
\frac{1}{\sin \varphi}\bigl(\frac{d}{d\theta} (\sin \varphi F_{\theta})
+\frac{d}{d\varphi} (\sin \varphi F_{\varphi})\bigr) 
=
\frac{dF_{\theta}}{d\theta} + \cot \varphi F_{\varphi} + \frac{d F_{\varphi}}{d \varphi}.
\end{equation}

An example of divergence-free  vector field (tangent to the sphere)
is, for an angle $\alpha$, is 
$$\vec{u} =u_{\varphi}\vec{\varphi}+u_{\theta}\vec{\theta}
=  - \sin \alpha \cos \theta \vec{\varphi} + (\cos \alpha+\sin \alpha \cot \varphi \sin \theta)\vec{\theta}.$$
The fact that this is divergence-free is evident from (\ref{div}).
Writing $x= \cos \theta \sin \varphi$, $y= \sin \theta \sin \varphi$ and $z=\cos\varphi$,
we have 
$$\vec{u}(x,y,z) 
=
\begin{pmatrix}
 -\sin\varphi \sin\theta \cos \alpha - \cos \varphi \sin \alpha\\ 
 \sin\varphi \cos\theta \cos \alpha  \\ 
 \sin \varphi \cos \theta \sin\alpha 
\end{pmatrix}
= 
\begin{pmatrix}
 -y \cos \alpha - z\sin \alpha\\ x\cos \alpha  \\ x \sin\alpha 
\end{pmatrix}
.$$

\paragraph{An example of a transport term}
We consider a first order operator of the form $\opL = \vec{u}\cdot \nabla$, where 
$\nabla$ is the ``surface'' gradient and $\vec{u} = u_\phi \vec{\phi} + u_{\theta} \vec{\theta}$ 
is a (tangent) vector field.
Carrying out the (Cartesian) inner product simply produces
$$\opL f = u_{\varphi}\frac{\partial f}{\partial\varphi}\langle \vec{\varphi},\vec{\varphi}\rangle +
u_{\theta}\frac{1}{\sin^2\varphi} \frac{\partial f}{\partial \theta} \langle \vec{\theta},\vec{\theta}\rangle 
=
u_{\varphi}\frac{\partial f}{\partial\varphi} +u_{\theta}  \frac{\partial f}{\partial \theta}
.$$

\subsection{First order operators}
  We apply this to a zonal kernel $\varPhi(x,y) = \phi(x\cdot y)$ in the first argument. 
So we can write the dot product as 
\begin{eqnarray*}
x\cdot y& =& 
\sin \varphi_1 \sin \varphi_2 \bigl(\cos \theta_1 \cos \theta_2 + \sin\theta_1 \sin \theta_2\bigr) 
+ \cos \varphi_1 \cos \varphi_2\\
&=&
\sin \varphi_1 \sin \varphi_2 \bigl(\cos (\theta_1 - \theta_2)\bigr) + \cos \varphi_1 \cos \varphi_2
\end{eqnarray*}
For first order differential problems,  we need $\opL^{(1)}\phi(x\cdot y)$,
which requires an expression for  $\phi'(t)$.
Indeed, note that
\begin{eqnarray}
\frac{\partial }{\partial \varphi}^{(1)} \phi(x\cdot y) &=& 
\left[\cos\varphi_1 \sin \varphi_2 
\bigl( \cos (\theta_1 - \theta_2)\bigr) - \sin \varphi_1 \cos \varphi_2\right]\phi'(x\cdot y)
= (\vec{\varphi}\cdot y) \phi'(x\cdot y)
\label{varphi_dot}\\
\frac{\partial }{\partial \theta}^{(1)} \phi(x\cdot y) &=&
-\sin \varphi_1 \sin \varphi_2 \bigl(\sin(\theta_1 - \theta_2)\bigr) \phi'(x\cdot y)
= (\vec{\theta}\cdot y) \phi'(x\cdot y)
\label{theta_dot}
\end{eqnarray}
From this, it follows that the surface gradient is
\begin{eqnarray}
\nabla^{(1)} \phi(x\cdot y) &=& 
(\vec{\varphi} \cdot y) \phi'(x\cdot y)
\vec{\varphi} + \frac{1}{\sin^2 \varphi}(\vec{\theta} \cdot y) \phi'(x\cdot y)
\vec{\theta}. \label{grad_dot}
\end{eqnarray}
For the transport term we get in this way the formula
$$\opL^{(1)} \phi(x \cdot y) = u_{\varphi} (\vec{\varphi} \cdot y) \phi'(x\cdot y) + u_{\theta} (\vec{\theta} \cdot y) \phi'(x\cdot y).$$

\subsection{Second order operators}\label{lb}
In principle, the kernel derivative formulas 
(\ref{varphi_dot}), (\ref{theta_dot}) and (\ref{grad_dot}) 
along with (\ref{div}) are sufficient
to calculate second order operators in divergence form $\opL =  \mathrm{div} (\vec{a} \nabla f)$
for a sufficiently smooth tensor field $\vec{a}$, although this may be too 
cumbersome to carry out by hand.

\paragraph{Laplace-Beltrami} 
For $\opL = \Delta$ (so $\vec{a} = \mathrm{Id}$), it is much easier to use the rotation invariance
of $\Delta$. In that case, we can write $x\cdot y = \cos(\vartheta)$, with $\vartheta$ the 
solid angle between $x$ and $y$ 
(equivalently, we can perform a rotation mapping $y$ to the north pole). 
In this case, rotation invariance gives $\Delta^{(1)} \phi(x\cdot y) =  \Delta \phi(\cos\vartheta) $
and
$$  \Delta \phi(\cos\vartheta) = 
\frac{1}{\sin \vartheta} \frac{\partial}{\partial \vartheta} 
\bigl( \sin \vartheta \frac{\partial}{\partial \vartheta} \phi(\cos \vartheta)\bigr)
=
\frac{1}{\sin \vartheta} \frac{\partial}{\partial \vartheta} 
\bigl(- \sin^2 \vartheta \phi'(\cos \vartheta)\bigr)
=
-2\cos\vartheta \phi'(\cos\vartheta) +(1-\cos^2\vartheta)\phi''(\cos\vartheta).
$$
which simplifies to
$$
\Delta^{(1)} \phi(x\cdot y) = (1-(x\cdot y)^2)\phi''(x\cdot y) -2(x\cdot y)\phi'(x\cdot y).$$

\paragraph{A second example:} In the basis $\{ \vec{\varphi}, \vec{\theta} \}$, 
we consider the tensor $\vec{a}$ given by
$$\vec{a}= a(\theta,\varphi) \begin{pmatrix} 1 & 0 \\ 0 & \sin^2\varphi \end{pmatrix}, $$
i.e. $\vec{a}$ lies parallel to the metric tensor of the unit sphere. 
In this case, we can simplify the second order differential operator $\opL$.
Because 
$\nabla f 
= \frac{\partial f}{\partial \theta} \vec{\varphi} 
  + \frac{1}{\sin^2\varphi} \frac{\partial f}{\partial \varphi} \vec{\theta}$,
we can write
\begin{eqnarray*} 
\opL f 
&=& \mathrm{div}(\vec{a} \nabla f) \\
&=& 
\frac{1}{\sin \varphi}
\frac{\partial}{\partial \varphi} \sin\varphi\, a(\theta,\varphi)  \frac{\partial f}{\partial \varphi}
+ \frac{\partial}{\partial \theta} a(\theta,\varphi) \frac{\partial f}{\partial \theta}\\
&=&
a(\theta,\varphi) \left( \Delta f - \cot^2 \varphi \frac{\partial^2 f}{\partial \theta^2}\right)    
+
\frac{\partial a}{\partial \varphi} \frac{\partial f}{\partial \varphi}+
\frac{\partial a}{\partial \theta}\frac{\partial f}{\partial \theta}
\end{eqnarray*}

For the kernel function $\phi(x\cdot y)$ centered at a fixed $y$ we therefore get, 
similarly as for the Laplace-Beltrami operator and the surface gradient the formula
\begin{eqnarray*}
 \opL^{(1)} \phi(x \cdot y) 
&=& 
a(\theta,\varphi) 
\left( 
  \Delta^{(1)} \phi(x \cdot y) 
  - 
  \cot^2 \varphi 
  \Big( 
      \phi''(x \cdot y) (\vec{\theta} \cdot y)^2 
      +  
      \phi'(x \cdot y) ( \frac{\partial \vec{\theta}}{\partial \theta} \cdot y) 
  \Big) 
\right) \\
&&
+
\left(
    \frac{\partial a}{\partial \varphi} \vec{\varphi} \cdot y
    +
    \frac{\partial a}{\partial \theta}\vec{\theta} \cdot y
\right)
\phi'(x \cdot y). 
\end{eqnarray*}

Considering this with $a(\varphi,\theta)= (1-\frac12 \cos \varphi)$ gives an elliptic PDE 
considered in \cite{NRW}. This simplifies to

$$ \opL^{(1)} \phi(x \cdot y)
= (1-\frac12 \cos \varphi) \left( \Delta^{(1)} \phi(x \cdot y) 
- \cot^2 \varphi \Big( \phi''(x \cdot y) (\vec{\theta} \cdot y)^2 +
  \phi'(x \cdot y) ( \frac{\partial \vec{\theta}}{\partial \theta} \cdot y) \Big) \right) 
+ \frac{1}{2} \sin{\varphi} \, (\vec{\varphi} \cdot y)  \phi'(x \cdot y), 
$$
where $(\varphi,\theta)$ are the spherical coordinates corresponding to $x$. 

\subsection{Derivatives of well known SBFs}
Given an SBF  of the form $\phi(t) = \psi(\epsilon R(t))$ (as described in \ref{S:kernel_examples}),
the first and second derivatives can be written as:
\begin{equation}\label{deriv_formulas}
\phi' = - \frac{\epsilon }{2R} \psi'(\epsilon R) =: \epsilon^2 w(\epsilon R)\qquad 
\text{and}
\qquad
 \phi'' = -\frac{\epsilon^3}{2R}w'(\epsilon R).
\end{equation}
\begin{table}[ht]
\centering
\begin{tabular}{|c  |c| c|}
\hline 
SBF
& 
$\phi'$
&
$\phi''$
\\
\hline
Wendland 2.5
&
$10 \epsilon^2 \left(1-\epsilon R\right)_+^3  \frac{1}{\epsilon R}$
& discontinuous
\\
\hline
Wendland 3.5
&
$
28 \epsilon^2 \left(1-\epsilon R\right)_+^5 \, 
\bigl(1 +5(\epsilon R)\bigr)
$
&
$420 \epsilon^4 \left(1-\epsilon R\right)_+^4$\\
\hline
Wendland  4.5
&
$11 \epsilon^2 \left(1-\epsilon R\right)_+^7 \, 
\bigl(1 +7(\epsilon R)+16(\epsilon R)^2\bigr)
$
&
$
132 \epsilon^4 \left(1-\epsilon R\right)_+^6\,
\bigl(1 +6(\epsilon R)\bigr) 
$
\\
\hline
Mat{\' e}rn  2.5
&
$\epsilon^2 e^{-\epsilon R} $
&
discontinuous
\\
\hline
Mat{\' e}rn 3.5&
 $\epsilon^2 e^{-\epsilon R} \,\bigl(1+\epsilon R(t)\bigr)$
 &
$\epsilon^4 e^{-\epsilon R}$
\\
\hline
Mat{\' e}rn 4.5&
$\epsilon^2 e^{-\epsilon R} \,\bigl(3 + 3(\epsilon R ) + (\epsilon R)^2 \bigr)$
&
$ \epsilon^4 e^{-\epsilon R} \,\bigl(1+\epsilon R\bigr)$
\\
\hline
Gaussian
& 
$ \epsilon^2 e^{- \bigl(\epsilon R\bigr)^2}$
&
$\epsilon^4 e^{- \bigl(\epsilon R\bigr)^2}$
\\
\hline
Inverse  multiquadric&
$
\frac{\beta \epsilon^2 }{\left(1+\bigl(\epsilon R\bigr)^2\right)^{\beta +1}}$
&
$\frac{\beta(\beta+1) \epsilon^4 }{\left(1+\bigl(\epsilon R\bigr)^2\right)^{\beta+2}}$
\\
\hline
\end{tabular}
\caption{Derivatives of SBFs appearing in Table 1.}
\end{table}

\bibliographystyle{plain}

%
\end{document}